\documentclass[10pt,oneside,english]{amsart}
\usepackage[T1]{fontenc}
\usepackage[latin9]{inputenc}
\usepackage{geometry}
\usepackage{amsthm}
\usepackage{amstext}
\usepackage{amssymb}

\makeatletter
\numberwithin{equation}{section}
\numberwithin{figure}{section}
  \theoremstyle{plain}
  \newtheorem*{thm*}{\protect\theoremname}
  \theoremstyle{plain}
  \newtheorem*{cor*}{\protect\corollaryname}
  \theoremstyle{plain}
  \newtheorem*{conjecture*}{\protect\conjecturename}
\theoremstyle{plain}
\newtheorem{thm}{\protect\theoremname}
  \theoremstyle{plain}
  \newtheorem{lem}[thm]{\protect\lemmaname}
  \theoremstyle{plain}
  \newtheorem{cor}[thm]{\protect\corollaryname}
  \theoremstyle{plain}
  \newtheorem{prop}[thm]{\protect\propositionname}
  \theoremstyle{remark}
  \newtheorem{rem}[thm]{\protect\remarkname}
  \theoremstyle{definition}
  \newtheorem{example}[thm]{\protect\examplename}

\usepackage{amsfonts}
\usepackage{xypic}
\theoremstyle{definition}

\makeatother

\usepackage{babel}
  \providecommand{\conjecturename}{Conjecture}
  \providecommand{\corollaryname}{Corollary}
  \providecommand{\examplename}{Example}
  \providecommand{\lemmaname}{Lemma}
  \providecommand{\propositionname}{Proposition}
  \providecommand{\remarkname}{Remark}
  \providecommand{\theoremname}{Theorem}
\providecommand{\theoremname}{Theorem}

\begin{document}
\subjclass[2010]{14R05, 14R25, 14E07}
\keywords{Rigid surfaces, Cancellation Problem}

\author{Adrien Dubouloz}

\address{IMB UMR5584, CNRS, Univ. Bourgogne Franche-Comté, F-21000 Dijon,
France.}

\email{adrien.dubouloz@u-bourgogne.fr}

\thanks{This research was partialy supported by ANR Grant \textquotedbl{}BirPol\textquotedbl{}
ANR-11-JS01-004-01. }

\title{Rigid affine surfaces with isomorphic $\mathbb{A}^{2}$-cylinders}
\begin{abstract}
We construct families of smooth affine surfaces with pairwise non
isomorphic $\mathbb{A}^{1}$-cylinders but whose $\mathbb{A}^{2}$-cylinders
are all isomorphic. These arise as complements of cuspidal hyperplane
sections of smooth projective cubic surfaces. 
\end{abstract}
\maketitle

\section*{Introduction}

The Zariski Cancellation Problem, which asks whether two, say smooth
affine, algebraic varieties $X$ and $Y$ with isomorphic cylinders
$X\times\mathbb{A}^{n}$ and $Y\times\mathbb{A}^{n}$ for some $n\geq1$
are isomorphic themselves, has been studied very actively during the
past decades culminating recently with a negative solution in dimension
$3$ and positive characteristic for the case $X=\mathbb{A}^{3}$
\cite{Gu14}. The situation in the complex case, and more generally
over any algebraically closed field of characteristic zero, is more
contrasted: cancellation is known to hold for curves \cite{AEH72}
and for $\mathbb{A}^{2}$ \cite{Fu79}, but many counter-examples
in every dimension higher or equal to $2$ have been discovered (see
\cite{Ru14} for a survey), inspired by the two pioneering constructions
of Hochster \cite{Ho72} and Danielewski \cite{Dan89}. 

Essentially all known families are counter-examples to the cancellation
of $1$-dimensional cylinders which arise from the existence of nontrivial
decompositions of certain locally trivial $\mathbb{A}^{2}$-bundles
over a base scheme $Z$. Namely, Hochster type constructions rely
on the existence of non free, $1$-stably free, projective modules
which in geometric term correspond to non trivial decompositions of
the trivial bundle $Z\times\mathbb{A}^{r+1}\rightarrow Z$ into a
trivial $\mathbb{A}^{1}$-bundle over a nontrivial vector bundle $E\rightarrow Z$
of rank $r\geq1$. For every such bundle, the varieties $X=E$ and
$Y=Z\times\mathbb{A}^{r}$ have isomorphic cylinders $X\times\mathbb{A}^{1}$
and $Y\times\mathbb{A}^{1}$, and one then gets a counter-example
to the cancellation problem whenever $X$ and $Y$, which by definition
are non isomorphic as schemes over $Z$, are actually non isomorphic
as abstract algebraic varieties \cite{Je10}. In contrast, Danielewski
type constructions usually involve non trivial decompositions of a
principal homogeneous $\mathbb{G}_{a}^{2}$-bundle $W\rightarrow Z$
into pairs $W\rightarrow X\rightarrow Z$ and $W\rightarrow Y\rightarrow Z$
consisting of trivial $\mathbb{G}_{a}$-bundles over nontrivial $\mathbb{G}_{a}$-bundles
$X\rightarrow Z$ and $Y\rightarrow Z$ with affine total spaces,
with the property that $W$ is isomorphic to the fiber product $W=X\times_{Z}Y$.
The isomorphism $X\times\mathbb{A}^{1}\simeq W\simeq Y\times\mathbb{A}^{1}$
is granted by definition, and similarly as in the previous type of
construction, one obtains counter-examples to the cancellation problem
as soon as $X$ and $Y$ are not isomorphic as abstract varieties. 

Non-cancellation phenomena with respect to higher dimensional cylinders
are more mysterious. In fact, it seems for instance that not even
a single explicit example of a pair of non-isomorphic varieties $X$
and $Y$ which fail the $\mathbb{A}^{2}$-cancellation property in
a minimal way, in the sense that $X\times\mathbb{A}^{2}$ and $Y\times\mathbb{A}^{2}$
are isomorphic while $X\times\mathbb{A}^{1}$ and $Y\times\mathbb{A}^{1}$
are still non isomorphic, is known so far. In this article, we fill
this gap by constructing a positive dimensional moduli of smooth affine
surfaces which fail the $\mathbb{A}^{2}$-cancellation property minimally.
That such varieties exist was certainly a natural expectation, and
their existence is therefore neither really surprising, nor probably
exciting in itself due to the abundance of simpler counter-examples
to the cancellation problem. Their interest lies rather in the fact
that they provide additional insight on the algebro-geometric properties
that a variety should satisfy in order to fail cancellation. 

Indeed, it follows from Iitaka-Fujita strong Cancellation Theorem
\cite{IiFu77} that a smooth affine variety $X$ which fails cancellation
has negative logarithmic Kodaira dimension, a property conjecturally
equivalent in dimension higher or equal to $3$ to the fact that $X$
is covered by images of the affine line and equivalent for surfaces
to the existence of an $\mathbb{A}^{1}$-fibration $\pi:X\rightarrow C$
over a smooth curve \cite{MS80}, i.e. a flat fibration with general
fibers isomorphic to the affine line. In the particular case of the
cancellation problem for $1$-dimensional cylinders, a further striking
discovery of Makar-Limanov is that the existence of nontrivial actions
of the additive group $\mathbb{G}_{a}$ on $X$ is a necessary condition
for non-cancellation. Namely, Makar-Limanov semi-rigidity theorem
\cite{MLNotes} (see also \cite[Proposition 9.23]{FreuBook}) asserts
that if $X$ is rigid, i.e. does not admit any nontrivial $\mathbb{G}_{a}$-action,
then the projection $\mathrm{pr}_{X}:X\times\mathbb{A}^{1}\rightarrow X$
is invariant under all $\mathbb{G}_{a}$-actions on $X$. As a consequence,
if either $X$ or $Y$ is rigid then every isomorphism between $X\times\mathbb{A}^{1}$
and $Y\times\mathbb{A}^{1}$ descends to an isomorphism between $X$
and $Y$. Combined with Fieseler's topological description of algebraic
quotient morphisms of $\mathbb{G}_{a}$-actions on smooth complex
affine surfaces \cite{Fie94}, these results imply that a smooth affine
surface which fails $\mathbb{A}^{1}$-cancellation must admit a nontrivial
$\mathbb{G}_{a}$-action whose algebraic quotient morphism $\pi:X\rightarrow X/\!/\mathbb{G}_{a}=\mathrm{Spec}(\Gamma(X,\mathcal{O}_{X})^{\mathbb{G}_{a}})$
is not a locally trivial $\mathbb{A}^{1}$-bundle. This holds of course
for the two smooth surfaces $xz-y(y+1)=0$ and $x^{2}z-y(y+1)=0$
used by Danielewski in his celebrated counter-example, showing a posteriori
that his construction was essentially optimal in this dimension. The
rich families of existing counter-examples to $\mathbb{A}^{1}$-cancellation
in dimension $2$ lend support to the conjecture that every smooth
affine surface which is neither rigid nor isomorphic to the total
space of line bundle over an affine curve fails the $\mathbb{A}^{1}$-cancellation
property. 

In view of this conjecture, a smooth affine surface $X$ which fails
the $\mathbb{A}^{2}$-cancellation property in a minimal way must
be simultaneously rigid and equipped with an $\mathbb{A}^{1}$-fibration
$\pi:X\rightarrow C$ over a smooth curve, and the well known fact
that $\mathbb{A}^{1}$-fibrations over affine curves are algebraic
quotient morphisms of nontrivial $\mathbb{G}_{a}$-actions implies
further that $C$ must be projective. This is precisely the case for
the family of surfaces we construct in this article, a particular
example being the smooth affine cubic surfaces 
\[
X=\{(-1+\alpha\sqrt[3]{2}x_{2}+\alpha\sqrt[3]{2}x_{3})^{3}+8(x_{1}^{3}+x_{2}^{3}+x_{3}^{3})=0\}\;\textrm{and\;}X'=\{(-1+2\alpha x_{1}+\alpha\sqrt[3]{2}x_{2}+\alpha\sqrt[3]{2}x_{3})^{3}+8(x_{1}^{3}+x_{2}^{3}+x_{3}^{3})=0\}
\]
in $\mathbb{A}^{3}=\mathrm{Spec}(\mathbb{C}[x_{1},x_{2},x_{3}])$,
where $\alpha=\exp(i\pi/3)$, which are both rigid and equipped with
an $\mathbb{A}^{1}$-fibration over $\mathbb{P}^{1}$. These arise
as the complements in the Fermat cubic surface $V=\{x_{0}^{3}+x_{1}^{3}+x_{2}^{3}+x_{3}^{3}=0\}$
in $\mathbb{P}^{3}$ of the plane cuspidal cubics $C=\{-(x_{2}+x_{3})^{3}+4(x_{1}^{3}+x_{2}^{3}+x_{3}^{3})\}=0$
and $C'=\{-((\sqrt[3]{2})^{2}x_{1}+x_{2}+x_{3})^{3}+4(x_{1}^{3}+x_{2}^{3}+x_{3}^{3})=0\}$
obtained by intersecting $V$ with its tangent hyperplane at the points
$p=[\alpha\sqrt[3]{2}:0:1:1]$ and $p'=[\alpha\sqrt[3]{2}:\sqrt[3]{2}:-1:1]$
respectively. The group of automorphisms of $V$ being isomorphic
to $\mathbb{Z}_{3}\times\mathfrak{S}_{4}$, where $\mathbb{Z}_{3}$
is the 3-torsion subgroup of $\mathrm{PGL(4;\mathbb{C})}$ and where
$\mathfrak{S}_{4}$ denotes the group of permutations of the variables,
the fact that $p$ and $p'$ do not belong to a same $\mathrm{Aut}(V)$-orbit
implies that the pairs $(V,C)$ and $(V,C')$ are not isomorphic.
Our main result just below then implies in turn that $X$ and $X'$
are non isomorphic, with isomorphic $\mathbb{A}^{2}$-cylinders $X\times\mathbb{A}^{2}$
and $X'\times\mathbb{A}^{2}$. 
\begin{thm*}
Let $(V_{i},C_{i})$, $i=1,2$, be non isomorphic pairs consisting
of a smooth cubic surface $V_{i}\subset\mathbb{P}^{3}$ and a cuspidal
hyperplane section $C_{i}=V_{i}\cap H_{i}$. Then the affine surfaces
$X_{i}=V_{i}\setminus C_{i}$ are non isomorphic, with non isomorphic
$\mathbb{A}^{1}$-cylinders $X_{i}\times\mathbb{A}^{1}$ but with
isomorphic $\mathbb{A}^{2}$-cylinders $X_{i}\times\mathbb{A}^{2}$,
$i=1,2$. 
\end{thm*}
As a consequence, all smooth affine surfaces arising as complements
of cuspidal hyperplane sections of smooth projective cubic surfaces
have isomorphic $\mathbb{A}^{2}$-cylinders. Noting that the projective
closure in $\mathbb{P}^{3}$ of the surface $X_{0}\subset\mathbb{A}^{3}=\mathrm{Spec}(\mathbb{C}[x,y,z])$
with equation $x^{2}y+y^{2}+z^{3}+1=0$ is a smooth cubic surface
intersecting the plane at infinity along the cuspidal cubic $x^{2}y+z^{3}=0$,
we obtain the following:
\begin{cor*}
Let $X$ be a smooth affine surface isomorphic to the complement of
a cuspidal hyperplane section of a smooth projective cubic surface.
Then $X\times\mathbb{A}^{2}$ is isomorphic to the affine cubic fourfold
$Z\subset\mathbb{A}^{5}=\mathrm{Spec}(\mathbb{C}[x,y,z][u,v])$ with
equation $x^{2}y+y^{2}+z^{3}+1=0$. Furthermore, $X$ is isomorphic
to the geometric quotient of a proper action of the group $\mathbb{G}_{a}^{2}$
on $Z$. \\

\end{cor*}
The scheme of the proof of the Theorem given in the next section is
the following. The fact that the affine surfaces $X_{1}$ and $X_{2}$
are non-isomorphic follows from the non-isomorphy of the pairs $(V_{1},C_{1})$
and $(V_{2},C_{2})$ via an argument of classical birational geometry
of projective cubic surfaces, which simultaneously renders the conclusion
that $X_{1}$ and $X_{2}$ are rigid. The non isomorphy of the cylinders
$X_{1}\times\mathbb{A}^{1}$ and $X_{2}\times\mathbb{A}^{1}$ is then
a straightforward consequence Makar-Limanov's semi-rigidity Theorem. 

The existence of an isomorphism between the $\mathbb{A}^{2}$-cylinders
$X_{1}\times\mathbb{A}^{2}$ and $X_{2}\times\mathbb{A}^{2}$ is derived
in two steps: the first one consists of another instance of a Danielewski
fiber product trick argument, which provides a smooth affine threefold
$W$ equipped with simultaneous structures of line bundles $\pi_{1}:W\rightarrow X_{1}$
and $\pi_{2}:W\rightarrow X_{2}$ over $X_{1}$ and $X_{2}$. But
here, in contrast with the situation in Danielewski's counter-example,
the fact that the $\mathbb{A}^{1}$-cylinders over $X_{1}$ and $X_{2}$
are not isomorphic implies that these two line bundles cannot be simultaneously
trivial. Nevertheless, the crucial observation which enables a second
step, reminiscent to Hochster construction, is that the pull-backs
via the isomorphisms $\pi_{i}^{*}:\mathrm{Pic}(X_{i})\rightarrow\mathrm{Pic}(W)$
of the classes of these lines bundles in the Picard groups of $X_{1}$
and $X_{2}$, say $L_{1}$ and $L_{2}$, coincide. Letting $q:E\rightarrow W$
be a line bundle representing the common inverse in $\mathrm{Pic}(W)$
of $\pi_{1}^{*}L_{1}=\pi_{2}^{*}L_{2}$, the composition $\pi_{i}\circ q:E\rightarrow X_{i}$
is then a vector bundle of rank $2$ isomorphic to the direct sum
$L_{i}\oplus L_{i}^{\vee}$, where $L_{i}^{\vee}$ denotes the dual
of $L_{i}$, hence isomorphic to $\det(E)\oplus\mathbb{A}_{X_{i}}^{1}=(L_{i}\otimes L_{i}^{\vee})\oplus\mathbb{A}_{X_{i}}^{1}\simeq\mathbb{A}_{X_{i}}^{1}\oplus\mathbb{A}_{X_{i}}^{1}$
by virtue of result of Pavaman Murthy \cite{Mur69} asserting that
every vector bundle on a smooth affine surface birationaly equivalent
to a ruled surface is isomorphic to the direct sum of a trivial bundle
with a line bundle. 

\[\xymatrix{X_1\times \mathbb{A}^2 \ar[dd]_{\mathrm{pr}_{S_1}} \ar[r]^{\sim} & E \ar[d]_q & X_2\times \mathbb{A}^2 \ar[dd]^{\mathrm{pr}_{X_2}} \ar[l]_{\sim} \\ & W \ar[dl]_{\pi_1} \ar[dr]^{\pi_2} & \\ X_1 & & X_2. }\]\\

The construction of these isomorphisms suggests the following strengthening
of the above conjecture characterizing smooth affine surfaces failing
the $\mathbb{A}^{1}$-cancellation property, which would settle the
question of the behavior of smooth affine surfaces under stabilization
by affine spaces:
\begin{conjecture*}
A smooth affine surface $X$ with negative logarithmic Kodaira dimension
is either isomorphic to the total space of a line bundle over a curve,
or it fails the $\mathbb{A}^{2}$-cancellation property. Furthermore,
every non rigid $X$ which fails the $\mathbb{A}^{2}$-cancellation
property also fails the $\mathbb{A}^{1}$-cancellation property.\\
 
\end{conjecture*}

\section{Proof of the  theorem}

\subsection{Rigid affine cubic surfaces }

Given a smooth cubic surface $V\subset\mathbb{P}^{3}$ and a hyperplane
section $V\cap H$ consisting of an irreducible plane cuspidal cubic
$C$, the restriction of the projection $\mathbb{P}^{3}\dashrightarrow\mathbb{P}^{2}$
from the singular point $p$ of $C$ induces a rational map $V\dashrightarrow\mathbb{P}^{2}$
of degree $2$ with $p$ as a unique proper base point. Its lift to
the blow-up $\alpha:Y\rightarrow V$ of $V$ at $p$ coincides with
the morphism $\theta:Y\rightarrow\mathbb{P}^{2}$ defined by the anti-canonical
linear system $|-K_{Y}|$ and it factors through a birational morphism
$Y\rightarrow Z$ to the anti-canonical model $Z=\mathrm{Proj}_{\mathbb{C}}(\bigoplus_{m\geq0}H^{0}(Y,-mK_{Y}))$
of $Y$, followed by a Galois double cover $Z\rightarrow\mathbb{P}^{2}$
ramified over a quartic curve. The nontrivial involution of the double
cover $Z\rightarrow\mathbb{P}^{2}$ lifts to a biregular involution
of $Y$ exchanging the proper transform of $C$ and the exceptional
divisor $E$ of $\alpha$. This involution descends back to a birational
map $G_{p}:V\dashrightarrow V$, called the \emph{Geiser involution
of }$V$ \emph{with center at} $p$, which contracts $C$ to $p$
and restricts to a biregular involution $j_{p}:X\rightarrow X$ of
the affine complement $X$ of $C$ in $V$. 
\begin{lem}
\label{lem:Bir-decomp}Let $X_{i}$ be the complements of cuspidal
hyperplanes sections $C_{i}=V_{i}\cap H_{i}$ with respective singular
points $p_{i}$ of smooth cubic surfaces $V_{i}\subset\mathbb{P}^{3}$,
$i=1,2$. Then for every isomorphism $\psi:X_{1}\stackrel{\sim}{\rightarrow}X_{2}$,
the birational map $\overline{\psi}:V_{1}\dashrightarrow V_{2}$ extending
$\psi$ is either an isomorphism of pairs $(V_{1},C_{1})\stackrel{\sim}{\rightarrow}(V_{2},C_{2})$
or it factors in a unique way as the composition of the Geiser involution
$G_{p_{1}}:V_{1}\dashrightarrow V_{1}$ followed by an isomorphism
of pairs $(V_{1},C_{1})\stackrel{\sim}{\rightarrow}(V_{2},C_{2})$.
In particular, $X_{1}$ and $X_{2}$ are isomorphic if and only if
so are the pairs $(V_{1},C_{1})$ and $(V_{2},C_{2})$. \end{lem}
\begin{proof}
Letting $\alpha_{i}:Y_{i}\rightarrow V_{i}$ be the blow-up of $V_{i}$
at $p_{i}$, with exceptional divisor $E_{i}$, $X_{i}$ is isomorphic
to $Y_{i}\setminus(C_{i}\cup E_{i})$ where we identified $C_{i}$
and its proper transform in $Y_{i}$. The birational map $\overline{\psi}:V_{1}\dashrightarrow V_{2}$
lifts to a birational $\overline{\Psi}=\alpha_{2}^{-1}\circ\overline{\psi}\circ\alpha_{1}:Y_{1}\dashrightarrow Y_{2}$
extending $\psi$, and the assertion is equivalent to the fact that
$\overline{\Psi}$ is an isomorphism of pairs $(Y_{1},C_{1}\cup E_{1})\stackrel{\sim}{\rightarrow}(Y_{2},C_{2}\cup E_{2})$.
Since $Y_{1}$ and $Y_{2}$ are smooth with the same Picard rank $\rho(Y_{i})=8$,
this holds provided that either $\overline{\Psi}$ or $\overline{\Psi}^{-1}$
is a morphism. So suppose for contradiction that $\overline{\Psi}$
or $\overline{\Psi}^{-1}$ are both strictly birational and let $Y_{1}\stackrel{\sigma_{1}}{\leftarrow}W\stackrel{\sigma_{2}}{\rightarrow}Y_{2}$
be the minimal resolution of $\overline{\Psi}$. Since $Y_{1}$ and
$Y_{2}$ are smooth and $\overline{\Psi}$ and $\overline{\Psi}^{-1}$
are both strictly birational, $\sigma_{1}$ consists of a non-empty
sequence of blow-ups of smooth points whose centers lie over $C_{1}\cup E_{1}$,
while $\sigma_{2}$ is a non-empty sequence of contractions of successives
$(-1)$-curves on $W$ supported on the total transform $\sigma_{1}^{-1}(C_{1}\cup E_{1})$
of $C_{1}\cup E_{1}$. Furthermore, the minimality assumption implies
that the first curve contracted by $\sigma_{2}$ is the proper transform
in $W$ of $C_{1}$ or $E_{1}$. Since $C_{1}$ and $E_{1}$ are $(-1)$-curves
in $Y_{1}$, the only possibility is thus that all successive centers
of $\sigma_{1}$ lie over $E_{1}\setminus C_{1}$ (resp. $C_{1}\setminus E_{1}$)
and that the first curve contracted by $\sigma_{2}$ is the proper
transform of $E_{1}$ (resp. $C_{1})$. But since $C_{1}$ and $E_{1}$
are tangent in $Y_{1}$, so are their proper transforms in $W$, and
then the image of $C_{1}$ (resp. $E_{1})$ by the contraction $\tau:W\rightarrow W'$
of $E_{1}$ (resp. $C_{1}$) factoring $\sigma_{2}:W\rightarrow Y_{2}$
would be singular. Since it cannot be contracted at any further step,
its image by $\sigma_{2}$ would be a singular curve contained in
$Y_{2}\setminus X_{2}=C_{2}\cup E_{2}$, which is absurd. \end{proof}
\begin{cor}
\label{cor:Rigidity} Let $X$ be the complement of a cuspidal hyperplane
section $C$ of a smooth cubic surface $V\subset\mathbb{P}^{3}$.
Then there exists a split exact sequence 
\[
0\rightarrow\mathrm{Aut}(V,C)\rightarrow\mathrm{Aut}(X)\rightarrow\{\mathrm{id}_{X},j_{p}\}\simeq\mathbb{Z}_{2}\rightarrow0,
\]
where $\mathrm{Aut}(V,C)$ is the automorphism group of the pair $(V,C)$
and $j_{p}:X\stackrel{\sim}{\rightarrow}X$ is the biregular involution
induced by the Geiser involution of $V$ with center at the singular
point $p$ of $C$. In particular, $\mathrm{Aut}(X)$ is a finite
group, isomorphic to $\mathbb{Z}_{2}$ for a general smooth cubic
surface $V$. \end{cor}
\begin{proof}
We view $\mathrm{Aut}(V,C)$ as a subgroup of $\mathrm{Aut}(X)$ via
the homomorphism which associates to every automorphism of $V$ preserving
$C$, hence $X$, its restriction to $X$. Since by virtue of the
previous lemma, the extension of every automorphism $\varphi$ of
$X$ to a birational self-map $\overline{\varphi}:V\dashrightarrow V$
is either an automorphism of the pair $(V,C)$ or the composition
of the Geiser involution $G_{p}:V\dashrightarrow V$ with an automorphism
of this pair, the first assertion follows. The second assertion is
a consequence of the fact that the automorphism group $\mathrm{Aut}(V)$
of a smooth cubic surface $V$ is always finite, actually trivial
for a general such surface. 
\end{proof}
The following proposition provides the first part of the proof of
the theorem:
\begin{prop}
\label{prop:Non-iso-A1} Let $X_{i}$ be the complements of cuspidal
hyperplanes sections $C_{i}=V_{i}\cap H_{i}$ of smooth cubic surfaces
$V_{i}\subset\mathbb{P}^{3}$, $i=1,2$. If the pairs $(V_{1},C_{1})$
and $(V_{2},C_{2})$ are not isomorphic then the $\mathbb{A}^{1}$-cylinders
$X_{1}\times\mathbb{A}^{1}$ and $X_{2}\times\mathbb{A}^{1}$ are
not isomorphic. \end{prop}
\begin{proof}
The rigidity of $X_{i}$ asserted by Corollary \ref{cor:Rigidity}
implies by virtue of \cite[Proposition 9.23]{FreuBook} that the Makar-Limanov
invariant $\mathrm{ML}(X_{i}\times\mathbb{A}^{1})$ of $X_{i}\times\mathbb{A}^{1}$
is equal to the sub-algebra $\Gamma(X_{i},\mathcal{O}_{X_{i}})$ of
$\Gamma(X_{i},\mathcal{O}_{X_{i}})[t]=\Gamma(X_{i}\times\mathbb{A}^{1},\mathcal{O}_{X_{i}\times\mathbb{A}^{1}})$.
Since every isomorphism between two algebras induces an isomorphism
between their Makar-Limanov invariants, it follows that every isomorphism
$X_{1}\times\mathbb{A}^{1}\stackrel{\sim}{\rightarrow}X_{2}\times\mathbb{A}^{1}$
descends to a unique isomorphism $\psi:X_{1}\stackrel{\sim}{\rightarrow}X_{2}$
making the following diagram commutative \[\xymatrix{X_1\times \mathbb{A}^1 \ar[r]^{\sim} \ar[d]_{\mathrm{pr}_{X_1}} & X_2 \times \mathbb{A}^1 \ar[d]^{\mathrm{pr}_{X_2}} \\ X_1 \ar[r]^{\psi} & X_2.}\] On
the other hand, the hypothesis that the pairs $(V_{1},C_{1})$ and
$(V_{2},C_{2})$ are not isomorphic combined with Lemma \ref{lem:Bir-decomp},
implies that $X_{1}$ is not isomorphic to $X_{2}$ and so, $X_{1}\times\mathbb{A}^{1}$
is not isomorphic to $X_{2}\times\mathbb{A}^{1}$. 
\end{proof}

\subsection{Isomorphisms between $\mathbb{A}^{2}$-cylinders}

As explained above, the first step of the construction is a Danielewski
fiber product trick creating a smooth affine threefold $W$ which
is simultaneously the total space of a line bundle over $X_{1}$ and
$X_{2}$. To setup such a fiber product argument, we first construct
a certain smooth algebraic space $\delta:B\rightarrow\mathbb{P}^{1}$
with the property that every complement $X$ of an irreducible cuspidal
hyperplane section $C$ of a smooth cubic surface $V\subset\mathbb{P}^{3}$
admits the structure of an \'etale locally trivial $\mathbb{A}^{1}$-bundle
$\rho:X\rightarrow B$.

\subsubsection{\noindent}

Letting $\mathbb{P}^{1}=\mathrm{Proj}(\mathbb{C}[z_{0},z_{1}])$,
the algebraic space $\delta:B\rightarrow\mathbb{P}^{1}$ is obtained
by the following gluing procedure:

1) We let $U_{\infty}=\mathbb{P}^{1}\setminus\{0\}\simeq\mathrm{Spec}(\mathbb{C}[w_{\infty}])$,
where $w_{\infty}=z_{1}/z_{0}$, and we let $\delta_{\infty}:B_{\infty}\rightarrow U_{\infty}$
be the scheme isomorphic to affine line with a $6$-fold origin obtained
by gluing six copies $\delta_{\infty,i}:U_{\infty,i}\stackrel{\sim}{\rightarrow}U_{\infty}$,
$i=1,\ldots,6$ of $U_{\infty}$, by the identity outside the points
$\infty_{i}=\delta_{\infty,i}^{-1}(\infty)$. 

2) We $U_{0}=\mathbb{P}^{1}\setminus\{\infty\}\simeq\mathrm{Spec}(\mathbb{C}[w_{0}])$,
where $w_{0}=z_{0}/z_{1}$, we let $\xi:\tilde{U}_{0}\simeq\mathbb{A}^{1}=\mathrm{Spec}(\mathbb{C}[\tilde{w}_{0}])\rightarrow U_{0}\simeq\mathrm{Spec}(\mathbb{C}[w_{0}])$,
$\tilde{w}_{0}\mapsto w_{0}=\tilde{w}_{0}^{3}$ be the triple Galois
cover totally ramified over $0$ and \'etale elsewhere, and we let
$\tilde{\delta}_{0}:\tilde{B}_{0}\rightarrow\tilde{U}_{0}$ be the
scheme isomorphic to the affine line with $3$-fold origin obtained
by gluing three copies $\tilde{U}_{0,1}$, $\tilde{U}_{0,\omega}$
and $\tilde{U}_{0,\omega^{2}}$ of $\tilde{U}_{0}$ by the identity
outside their respective origins $\tilde{0}_{0,1}$, $\tilde{0}_{0,\omega}$
and $\tilde{0}_{0,\omega^{2}}$. The action of the Galois group $\mu_{3}=\{1,\omega,\omega^{2}\}$
of complex third roots of unity of the covering $\xi$ lifts to fixed
point free action on $\tilde{B}_{0}$ given locally by $\tilde{U}_{0,\eta}\ni\tilde{z}_{0}\mapsto\omega\tilde{z}_{0}\in\tilde{U}_{0,\omega\eta}$.
Since the latter has trivial isotropies, a geometric quotient exists
in the category of algebraic spaces in the form an \'etale locally
trivial $\mu_{3}$-bundle $\tilde{B}_{0}\rightarrow\tilde{B}_{0}/\mu_{3}=B_{0}$
over a certain algebraic space $B_{0}$. Furthermore, the $\mu_{3}$-equivariant
morphism $\tilde{\delta}_{0}:\tilde{B}_{0}\rightarrow\tilde{U}_{0}$
descends to a morphism $\delta_{0}:B_{0}\rightarrow\tilde{U}_{0}/\negmedspace/\mu_{3}\simeq U_{0}$
restricting to an isomorphism over $U_{0}\setminus\{0\}$ and totally
ramified over $\{0\}$, with ramification index $3$. 

3) Finally, $\delta:B\rightarrow\mathbb{P}^{1}$ is obtained by gluing
$\delta_{\infty}:B_{\infty}\rightarrow U_{\infty}$ and $\delta_{0}:B_{0}\rightarrow U_{0}$
along the open sub-schemes $\delta_{\infty}^{-1}(U_{0}\cap U_{\infty})\simeq\mathrm{Spec}(\mathbb{C}[w_{\infty}^{\pm1}])$
and $\delta_{0}^{-1}(U_{0}\cap U_{\infty})\simeq\mathrm{Spec}(\mathbb{C}[w_{\infty}^{\pm1}])$
by the isomorphism $w_{\infty}\mapsto w_{0}=w_{\infty}^{-1}$. 
\begin{rem}
Letting $p_{0}$ be the unique closed point of $B$ over $0\in U_{0}\subset\mathbb{P}^{1}$,
we have $\delta^{-1}(0)=3p_{0}$ while the restriction of $\delta$
over $U_{0}\setminus\{0\}$ is an isomorphism. This implies that $B$
is not a scheme, for otherwise the restriction of $\delta$ over $U_{0}$
would be an isomorphism by virtue of Zariski Main Theorem. In fact,
$p_{0}$ is a point which does not admit any affine open neighborhood
$V$: otherwise the inverse image of $V$ by the finite morphism $\tilde{B}_{0}\rightarrow\tilde{B}_{0}/\mu_{3}=B_{0}$
would be an affine open sub-scheme of $\tilde{B}_{0}$ containing
the three points $\tilde{0}_{0,1}$, $\tilde{0}_{0,\omega}$ and $\tilde{0}_{0,\omega^{2}}$
which is impossible. 
\end{rem}

\subsubsection{\noindent}

Since the automorphism group of $\mathbb{A}^{1}$ is the affine group
$\mathrm{Aff}_{1}=\mathbb{G}_{m}\ltimes\mathbb{G}_{a}$, every \'etale
locally trivial $\mathbb{A}^{1}$-bundle $\rho:S\rightarrow B$ is
an affine-linear bundle whose isomorphy class is determined by an
element in the non-abelian cohomology group $H_{\mathrm{\acute{e}t}}^{1}(B,\mathrm{Aff}_{1})$.
Equivalently $\rho:S\rightarrow B$ is a principal homogeneous bundle
under the action of a line bundle $L\rightarrow B$, considered as
a locally constant group scheme over $B$ for the group law induced
by the addition of germs of sections, whose class in $\mathrm{Pic}(B)\simeq H_{\mathrm{\acute{e}t}}^{1}(B,\mathbb{G}_{m})$
coincides with the image of the isomorphy class of $\rho:S\rightarrow B$
in $H_{\mathrm{\acute{e}t}}^{1}(B,\mathrm{Aff}_{1})$ by the map $H_{\mathrm{\acute{e}t}}^{1}(B,\mathrm{Aff}_{1})\rightarrow H_{\mathrm{\acute{e}t}}^{1}(B,\mathbb{G}_{m})$
in the long exact sequence of non-abelian cohomology 

\[
0\rightarrow H^{0}\left(B,\mathbb{G}_{a}\right)\rightarrow H^{0}\left(B,\mathrm{Aff}_{1}\right)\rightarrow H^{0}\left(B,\mathrm{Aff}_{1}\right)\rightarrow H_{\mathrm{\acute{e}t}}^{1}\left(B,\mathbb{G}_{a}\right)\rightarrow H_{\mathrm{\acute{e}t}}^{1}\left(B,\mathrm{Aff}_{1}\right)\rightarrow H_{\mathrm{\acute{e}t}}^{1}\left(B,\mathbb{G}_{m}\right)
\]
associated to the short exact sequence $0\rightarrow\mathbb{G}_{a}\rightarrow\mathrm{Aff}_{1}\rightarrow\mathbb{G}_{m}\rightarrow0$.
Isomorphy classes of principal homogeneous under a given line bundle
$L\rightarrow B$ are are in turn classified by the cohomology group
$H_{\mathrm{\acute{e}t}}^{1}(B,L)$.
\begin{prop}
\label{prop:A1-fibration} The complement $X$ of a cuspidal hyperplane
section $C$ of a smooth cubic surface $V\subset\mathbb{P}^{3}$ admits
an $\mathbb{A}^{1}$-fibration $f:X\rightarrow\mathbb{P}^{1}$ which
factors through a principal homogeneous bundle $\rho:X\rightarrow B$
under the action of the cotangent line bundle $\gamma:\Omega_{B}^{1}\rightarrow B$
of $B$. \end{prop}
\begin{proof}
Since $C$ is an anti-canonical divisor on $V$, it follows from adjunction
formula that every line on $V$ intersects $C$ transversally in a
unique point. The image of $C$ by the contraction $\tau:V\rightarrow\mathbb{P}^{2}$
of any $6$-tuple of disjoint lines, $L_{1},\ldots,L_{6}$, on $V$
is therefore a rational cuspidal cubic containing the images $q_{i}=\tau(L_{i})$,
$i=1,\ldots,6$, in its regular locus. The rational pencil $\mathbb{P}^{2}\dashrightarrow\mathbb{P}^{1}$
generated by $\tau(C)$ and three times its tangent line $T$ at its
singular point lifts to a rational pencil $\overline{f}:V\dashrightarrow\mathbb{P}^{1}$
whose restriction to $X$ is an $\mathbb{A}^{1}$-fibration $f:X\rightarrow\mathbb{P}^{1}$
with two degenerate fibers: one irreducible of multiplicity three
consisting of the intersection of the proper transform of $T$ with
$X$, and a reduced one consisting of the disjoint union of the curves
$L_{i}\cap X\simeq\mathbb{A}^{1}$, $i=1,\ldots,6$. Choosing homogeneous
coordinates $[z_{0}:z_{1}]$ on $\mathbb{P}^{1}$ in such a way that
$0$ and $\infty$ are the respective images of $T$ and $C$ by $\overline{f}$,
the same argument as in \cite[§4]{DubFlex15} implies that $f:X\rightarrow\mathbb{P}^{1}$
factors through an \'etale locally trivial $\mathbb{A}^{1}$-bundle
$\rho:X\rightarrow B$. Letting $\gamma:L\rightarrow B$ be the line
bundle under which $\rho:X\rightarrow B$ becomes a principal homogeneous
bundle, it follows from the relative cotangent exact sequence 
\[
0\rightarrow\rho^{*}\Omega_{B}^{1}\rightarrow\Omega_{X}^{1}\rightarrow\Omega_{X/B}^{1}\simeq\rho^{*}L^{\vee}\rightarrow0
\]
of $\rho$ that $\det\Omega_{X}^{1}\simeq\rho^{*}(\Omega_{B}^{1}\otimes L^{\vee})$.
Since $\det\Omega_{X}^{1}$ is trivial as $C$ is an anti-canonical
divisor on $V$ and since $\rho^{*}:\mathrm{Pic}(B)\rightarrow\mathrm{Pic}(X)$
is an isomorphism because $\rho:X\rightarrow B$ is a locally trivial
$\mathbb{A}^{1}$-bundle, we conclude that $L\simeq\Omega_{B}^{1}$. \end{proof}
\begin{rem}
By construction of $\delta:B\rightarrow\mathbb{P}^{1}$, we have $\delta^{-1}(0)=3p_{0}$
and $\delta^{-1}(\infty)=\sum_{i=1}^{6}\infty_{i}$. The Picard group
$\mathrm{Pic}(B)$ of $B$ is thus isomorphic to $\mathbb{Z}^{6}$
generated by the classes of the lines bundle $\mathcal{O}_{B}(p_{0})$,
$\mathcal{O}_{B}(\infty_{i})$, $i=1,\ldots,6$, with the unique relation
$\mathcal{O}_{B}(3p_{0})=\mathcal{O}_{B}(\sum_{i=1}^{6}\infty_{i})$.
Furthermore, since $\delta$ is \'etale except at $p_{0}$ where
it has ramification index $3$, we deduce from the ramification formula
for the morphism $\delta:B\rightarrow\mathbb{P}^{1}$ that the cotangent
bundle $\gamma:\Omega_{B}^{1}\rightarrow B$ of $B$ is isomorphic
to 
\[
\delta^{*}\Omega_{\mathbb{P}^{1}}\otimes_{\mathcal{O}_{B}}\mathcal{O}_{B}(2p_{0})\simeq\delta^{*}(\mathcal{O}_{\mathbb{P}^{1}}(-\{0\}-\{\infty\}))\otimes_{\mathcal{O}_{B}}\mathcal{O}_{B}(2p_{0})\simeq\mathcal{O}_{B}(-p_{0}-\sum_{i=1}^{6}\infty_{i}).
\]

\end{rem}

\subsubsection{\noindent}

Now we are ready for the second step of the construction, which completes
the proof of the theorem. Letting $X_{i}$, $i=1,2$, be the complements
of irreducible plane cuspidal hyperplanes sections $C_{i}=V_{i}\cap H_{i}$
of smooth cubic surfaces $V_{i}\subset\mathbb{P}^{3}$, Proposition
\ref{prop:A1-fibration} asserts the existence of principal homogeneous
bundles $\rho_{i}:X_{i}\rightarrow B$ under the action of the cotangent
line bundle $\gamma:\Omega_{B}^{1}\rightarrow B$ of $B$. The fiber
product $W=X_{1}\times_{B}X_{2}$ inherits via the first and second
projections respectively the structure of a principal homogeneous
bundle $\pi_{i}:W\rightarrow X_{i}$ under $\rho_{i}^{*}\Omega_{B}^{1}$,
$i=1,2$. Since $X_{i}$ is affine, the vanishing of $H_{\mathrm{\acute{e}t}}^{1}(X_{i},\rho_{i}^{*}\Omega_{B}^{1})$
implies that these bundles are both trivial, yielding isomorphisms
$\rho_{1}^{*}\Omega_{B}^{1}\simeq W\simeq\rho_{2}^{*}\Omega_{B}^{1}$.
Letting $q:E\rightarrow W$ be the pull-back of the dual $(\Omega_{B}^{1})^{\vee}$
of $\Omega_{B}^{1}$ by the morphism $\rho_{1}\circ\pi_{1}=\rho_{2}\circ\pi_{2}:W\rightarrow B$,
$\pi_{i}\circ q:E\rightarrow S_{i}$ is a vector bundle over $X_{i}$
isomorphic to the direct sum of $\rho_{i}^{*}\Omega_{B}^{1}$ and
$\rho_{i}^{*}(\Omega_{B}^{1})^{\vee}$: \[\xymatrix{ X_1 \times \mathbb{A}^2 \simeq_{X_1} \rho_1^*\Omega^1_B \oplus \rho_1^*(\Omega^1_B)^{\vee} \ar[d] \ar[r]^-{\sim} & E \ar[d]_{q}  & \rho_2^*\Omega^1_B \oplus \rho_2^*(\Omega^1_B)^{\vee} \simeq_{X_2} X_1\times \mathbb{A}^2 \ar[d] \ar[l]_-{\sim} \\ 
\rho_1^* \Omega^1_B \ar[r]^{\sim} \ar[d] & X_1\times_B X_2 \ar[dl]_{\pi_1} \ar[dr]^{\pi_2} & \rho_2^* \Omega^1_B \ar[d] \ar[l]_{\sim} \\ X_1 \ar[dr]^{\rho_1} & & X_2 \ar[dl]_{\rho_2} \\  & B & }\] So by virtue of \cite[Theorem 3.1]{Mur69}, $E$ is isomorphic as
a vector bundle over $X_{i}$ to $\det(\rho_{i}^{*}\Omega_{B}^{1}\oplus\rho_{i}^{*}(\Omega_{B}^{1})^{\vee})\oplus\mathbb{A}_{X_{i}}^{1}\simeq\mathbb{A}_{X_{i}}^{1}\oplus\mathbb{A}_{X_{i}}^{1}$
providing the desired isomorphisms $X_{1}\times\mathbb{A}^{2}\simeq E\simeq X_{2}\times\mathbb{A}^{2}$. 
\begin{example}
Let $V\subset\mathbb{P}^{3}$ be a general smooth cubic surface and
let $\Delta\subset V$ be the curve consisting of points $p$ of $V$
at which the projective tangent hyperplane $T_{p}V\subset\mathbb{P}^{3}$
of $V$ at $p$ intersects $V$ along a cuspidal cubic. Let $\mathcal{V}=\Delta\times V$
and let $\mathcal{C}\subset\mathcal{V}$ be relatively ample Cartier
divisor with respect to $\mathrm{pr}_{\Delta}:\mathcal{V}\rightarrow\Delta$
whose fiber $\mathcal{C}_{p}$ over every point $p\in\Delta$ is equal
to the intersection $C_{p}=V\cap T_{p}V$. Since $\mathrm{Aut}(V)$
is trivial, the pairs $(V,C_{p})$, $p\in\Delta$, are pairwise non
isomorphic, and so $\Theta=\mathrm{pr}_{\Delta}\mid_{\mathcal{X}}:\mathcal{X}=\mathcal{V}\setminus\mathcal{C}\rightarrow\Delta$
is a family of pairwise non isomorphic rigid smooth affine surfaces
whose $\mathbb{A}^{2}$-cylinders are all isomorphic. \end{example}
\begin{rem}
The $\mathbb{A}^{2}$-cylinder $X\times\mathbb{A}^{2}$ over the complement
$X$ of a cuspidal hyperplane section $C$ of a smooth cubic surface
$V$ is \emph{flexible in codimension} $1$, that is, for every closed
point $p$ outside a possible empty closed subset $Z\subset X\times\mathbb{A}^{2}$
of codimension at least two, the tangent space $T_{X\times\mathbb{A}^{2},p}$
of $X\times\mathbb{A}^{2}$ at $p$ is spanned by tangent vectors
to orbits of algebraic $\mathbb{G}_{a}$-actions on $X\times\mathbb{A}^{2}$.
This can be seen as follows: one first constructs by a similar procedure
as in \cite[§3.2]{DubFlex15} a flexible mate $S$ for $X$, in the
form of smooth affine surface flexible in codimension $1$ admitting
an $\mathbb{A}^{1}$-fibration $\pi:S\rightarrow\mathbb{P}^{1}$ which
factors through a principal homogeneous bundle $\tilde{\pi}:S\rightarrow B$
under the action of a certain line bundle $\gamma':L\rightarrow B$.
The fiber product $S\times_{B}X$ is then a smooth affine threefold
which is simultaneously isomorphic to the total spaces of the line
bundles $\tilde{\pi}^{*}\Omega_{B}^{1}$ and $\rho^{*}L$ over $S$
and $X$ via the first and second projection respectively. Since $S$
is flexible in codimension $1$, it follows from \cite[Lemma 2.3]{DubFlex15}
that $S\times_{B}X$ and the total space $F\rightarrow S\times_{B}X$
of the pull-back of $L^{\vee}$ by the morphism $\tilde{\pi}\circ\mathrm{pr}_{S}=\rho\circ\mathrm{pr}_{X}$
are both flexible in codimension $1$. By construction, $F$ is a
vector bundle of rank $2$ over $X$, isomorphic to $\rho^{*}(L\oplus L^{\vee})$
hence to the trivial vector bundle $X\times\mathbb{A}^{2}$ by virtue
of \cite[Theorem 3.1]{Mur69}. 
\end{rem}
\bibliographystyle{amsplain}

\end{document}